\documentclass[12pt,twoside,a4paper]{amsart}

\usepackage[margin=2.5cm]{geometry}

\usepackage{amssymb,mathtools}
\usepackage[dvipsnames]{xcolor}
\usepackage{hyperref}

\usepackage{float}

\newtheorem{thm}{Theorem}[section]
\newtheorem{lem}[thm]{Lemma}
\newtheorem{prop}[thm]{Proposition}
\newtheorem{cor}[thm]{Corollary}
\theoremstyle{definition}
\newtheorem{defn}[thm]{Definition}

\newtheorem{conjecture}[thm]{Conjecture}

\numberwithin{equation}{section}
\allowdisplaybreaks

\usepackage[color=Orange!50!white, textsize=scriptsize]{todonotes}

\newcommand{\RR}{\mathbb{R}}

\begin{document}

\keywords {Kakeya problem, polynomial method, projection theory}
\subjclass[2010]{28A80 (primary), 44A12 (secondary)}

\title[On the Peres--Schlag orthogonal projection problem]{On the Peres--Schlag orthogonal projection problem \\and Kakeya-type sets}

\author{Guo-Dong Hong} 
\address{Department of Mathematics, California Institute of Technology, Pasadena, CA 91125,
USA}
\email{ghong@caltech.edu}

\author{Chong-Wei Liang}
\address{Department of Mathematics, National Taiwan University, Taiwan.}
\email{d10221001@ntu.edu.tw}

\author{Chun-Yen Shen} 
\address{Department of Mathematics, National Taiwan University, Taiwan.}
\email{cyshen@math.ntu.edu.tw}
\date{}

\begin{abstract}
We investigate the Peres--Schlag nonempty interior problem for orthogonal projections in both the finite-field and Euclidean settings. Over finite fields $\mathbb F_q^n$, we employ the polynomial method to establish sharp projection results, and uncover a new connection with stability versions of the finite-field \((n,m)\)-set problem. Over Euclidean spaces $\mathbb R^n$, we obtain improved nonempty interior results beyond those of Peres and Schlag in certain parameter ranges. Our proof combines techniques from geometric measure theory and harmonic analysis, including $L^p$-estimates for Kakeya maximal operators and maximal $k$-plane transforms.
\end{abstract}
\maketitle

\section{Introduction}

A central problem in geometric measure theory is to understand how the size and structure of a set are reflected in its orthogonal projections \cite{MR3558147,MR0248779,MR0409774,MR0063439}. The classical Marstrand--Mattila projection theorem asserts that if \(E \subseteq \mathbb{R}^n\) has Hausdorff dimension \(s\), then for almost every \(l\)-dimensional subspace \(V \in G(n,l)\), the orthogonal projection of \(E\) onto \(V\) has Hausdorff dimension \(\min(s,l)\). A more delicate question asks when typical projections have a nonempty interior. In this direction, Peres and Schlag \cite[Corollary 6.2]{MR1749437} proved the following theorem.

\begin{thm}
\label{thm:peres-schlag-euclidean}
Let \(n,l \in \mathbb{N}\), let \(s>0\), and let \(E \subseteq \mathbb{R}^n\) be a Borel set with $\dim_{\mathrm{H}} E = s$.
If \(s>2l\), then for almost every \(V \in G(n,l)\), the orthogonal projection of \(E\) onto \(V\) has a nonempty interior.
\end{thm}

When \(l=1\), this threshold is sharp. Indeed, Besicovitch line-set constructions in the plane, combined with a rational-translate argument of Mattila \cite{MR2044636} (see also \cite{MR1749437}), produce two-dimensional sets with the property that every projection onto lines contains no intervals. Higher-dimensional examples at the same threshold may similarly be obtained by taking products with Euclidean spaces.
For projections onto \(l\)-planes with \(l \geq 2\), however, the sharp threshold remains unclear. Peres--Schlag \cite[Chapter 6]{MR1749437} observed that the exponent \(2l\) may not be optimal once \(l \geq 2\).

The present paper develops the Peres--Schlag projection problem in both finite-field and Euclidean settings. We first study finite-field analogues of the projection problem. Our results show that the sharp threshold over finite fields is \(l+1\), rather than the Peres--Schlag exponent \(2l\). The argument reveals a new connection between projection problems and finite-field Kakeya-type sets \cite{MR2525780,MR3143848,MR2604979}.

Next, we return to the Euclidean setting. We observe that the argument of Peres and Schlag \cite{MR1749437} may be reformulated in terms of \(L^2\)-bounds for maximal \(k\)-plane transforms controlled by Riesz energy. Combining this observation with the framework of Bourgain \cite{MR1097257} and the current best Kakeya maximal estimates \cite{BCCLXZ,MR4521046,MR3868003,MR4201413,MR1363209}, we obtain improved nonempty interior results for orthogonal projections in certain parameter ranges.

\subsection{Finite-field setting}

Throughout this paper, \(q\) denotes a prime power and \(\mathbb{F}_q\) denotes the finite field with \(q\) elements.
We write \(G(n,l)(\mathbb{F}_q)\) for the Grassmannian of \(l\)-dimensional linear subspaces of \(\mathbb{F}_q^n\). For \(V \in G(n,l)(\mathbb{F}_q)\), define the projection
\[
\Pi_V : \mathbb{F}_q^n \longrightarrow \mathbb{F}_q^n/V^\perp\quad\text{by}\quad \Pi_V(x) := x + V^\perp,
\]
where \(V^\perp\) denotes the orthogonal complement of \(V\) with respect to the standard bilinear form on \(\mathbb{F}_q^n\). Thus, \(\Pi_V(E)\) is the set of cosets of \(V^\perp\) that intersect \(E\). We say that \(\Pi_V(E)\) is \emph{full} if
\begin{align}\label{analoguesfull}
\Pi_V(E)=\mathbb{F}_q^n/V^\perp.
\end{align}
Since $
\mathbb{F}_q^n/V^\perp \cong \mathbb{F}_q^l$,
(\ref{analoguesfull}) is the finite-field analogue of saying that the projection has a nonempty interior.

Finite-field analogues of projection theorems over prime fields were first studied by Chen \cite{MR3753167}. Chen proved a finite-field analogue of the Peres--Schlag \(2l\)-threshold: if a set \(E \subseteq \mathbb{F}_p^n\) satisfies
\[
|E| \gtrsim p^s\quad\text{with}\quad s>2l,
\] then for almost every \(l\)-dimensional finite-field projection, the image of \(E\) is full. Thus, Chen's result gives a finite-field version of the Peres--Schlag exponent \(2l\).

Our first main result substantially improves this threshold.

\begin{thm}
\label{thm:main-projection}
Let \(2\le l<n\), and let \(E\subseteq\mathbb F_q^n\). Suppose that
\[
|E|\gtrsim q^s,\quad\text{for some}\,\,s>l+1.
\] Then for \(1-o_q(1)\) many \(V\in G(n,l)(\mathbb F_q)\),
\[
\Pi_V(E)=\mathbb F_q^n/V^\perp.
\]
 Equivalently,
\[
\left|
\left\{
V\in G(n,l)(\mathbb F_q):
\Pi_V(E)\neq \mathbb F_q^n/V^\perp
\right\}
\right|
=
o_q(1)\,|G(n,l)(\mathbb F_q)|.
\]
\end{thm}
For every \(l\ge2\), this yields a genuine improvement. Moreover, the exponent \(l+1\) is the best possible.

\begin{prop}
\label{thm:projection-sharpness}
Let \(2\le l<n-1\). Then there exists a set \(E\subseteq\mathbb F_q^n\) such that
\[
|E|
=
(1-2^{-l})q^{\,l+1}+O(q^l)
\]
and
\[
\Pi_V(E)\neq \mathbb F_q^n/V^\perp
\]
for every \(V\in G(n,l)(\mathbb F_q)\). In particular, the exponent \(l+1\) in Theorem~\ref{thm:main-projection} cannot be improved.
\end{prop}

The sharpness phenomenon is closely related to the finite-field \(m\)-plane Kakeya problem, where \(m=n-l\). Indeed, for a fixed \(V\in G(n,l)(\mathbb F_q)\), the projection \(\Pi_V(E)\) is not full if and only if there exists a coset of \(V^\perp\) that is disjoint from \(E\). Equivalently, if we write
\[
K:=\mathbb F_q^n\setminus E,
\]
then \(K\) contains an affine \(m\)-plane whose direction is \(V^\perp\).

Thus, a set whose \(l\)-dimensional projections are never full is precisely the complement of a set containing an affine \(m\)-plane in every \(m\)-dimensional direction. In other words, such examples arise naturally as complements of \(m\)-plane Kakeya sets. Applying the Ellenberg--Oberlin--Tao construction \cite{MR2604979} yields \(m\)-plane Kakeya sets whose complements have cardinality
\[
(1-2^{-l})q^{\,l+1}+O(q^l),
\]
which immediately gives Proposition~\ref{thm:projection-sharpness}.

The proof of Theorem~\ref{thm:main-projection} is also guided by this Kakeya viewpoint. If
\[
\Pi_V(E)\neq \mathbb F_q^n/V^\perp,
\]
then some coset of \(V^\perp\) is disjoint from \(E\), and therefore lies entirely inside
\[
K=\mathbb F_q^n\setminus E.
\]
Consequently, if the projection fails to be full for an \(\alpha\)-proportion of the directions \(V\in G(n,l)(\mathbb F_q)\), then the complement \(K\) contains a translate of an \((n-l)\)-dimensional affine plane in an \(\alpha\)-proportion of all \((n-l)\)-dimensional directions.

This observation reduces the projection problem to a quantitative Kakeya-type question. It naturally leads to the following stability version of the finite-field \(m\)-plane Kakeya problem, which will be the main ingredient in the proof of Theorem~\ref{thm:main-projection}.

\begin{defn}
\label{def:alpha-almost-m-plane-kakeya}~
Let \(1\le m<n\), and let \(0<\alpha\le 1\). For a set
\(K\subseteq \mathbb F_q^n\), define
\[
    S_K
    :=
    \left\{
        W\in G(n,m)(\mathbb F_q):
        \text{there exists }a\in \mathbb F_q^n
        \text{ such that }a+W\subseteq K
    \right\}.
\]
We say that \(K\) is an \(\alpha\)-almost \(m\)-plane Kakeya set if
\[
    |S_K|\ge \alpha |G(n,m)(\mathbb F_q)|.
\]
\end{defn}

Our next theorem bounds the size of an \(\alpha\)-almost \(m\)-plane Kakeya set. The proof follows the multiplicity-polynomial method of Dvir \cite{MR2525780}, Dvir--Kopparty--Saraf--Sudan \cite{MR3143848} and Ellenberg--Oberlin--Tao \cite{MR2604979}.  The novelty is that one only assumes the Kakeya condition
in a positive fraction of all \(m\)-dimensional directions. The additional ingredient is a
stability version of their projective vanishing lemma: a low-degree homogeneous
polynomial cannot vanish identically on too large a fraction of
\(m\)-dimensional directions. See also \cite{MR3004132,MR3272924,MR0575692} for more material on polynomial methods.

\begin{thm}
\label{thm:alpha-almost-kakeya}~
Let \(1\le m<n\), let \(0<\alpha\le 1\), and let
\(K\subseteq \mathbb F_q^n\).  Suppose that \(K\) is an
\(\alpha\)-almost $m$-plane Kakeya set.  Then
\[
    |K|
    \ge
    q^n-C_{n,m}\alpha^{-1}q^{n-m+1},
\]
where \(C_{n,m}\) depends only on \(n\) and \(m\).
\end{thm}

When \(\alpha=1\), Theorem~\ref{thm:alpha-almost-kakeya} recovers the same error
order as the Ellenberg--Oberlin--Tao \(m\)-plane Kakeya lower bound, though not
the sharp constant.

\subsection{Continuous setting}

We denote by \(G(n,k)\) the set of \(k\)-dimensional linear subspaces of
\(\mathbb R^n\), and by \(\gamma_{n,k}\) its unique orthogonally invariant Borel
probability measure. A Borel set \(B\subset\mathbb R^n\) is called an
\emph{\((n,k)\)-Besicovitch set} if it has Lebesgue measure zero and, for every
\(V\in G(n,k)\), there exists \(a\in\mathbb R^n\) such that
\[
    B^n_1(a)\cap (V+a)\subset B,
\]
where $B^n_1(a):=\{x \in \RR^n: \|x-a\| \leq 1\}$.
Peres and Schlag observed in \cite{MR1749437} that their
nonempty-interior projection theorem is closely related to the nonexistence problem
for \((n,k)\)-Besicovitch sets. Their Corollary~6.2 recovers Falconer's
nonexistence theorem \cite{MR0553579} in the range \(k>n/2\). They also point out that Bourgain's
theorem gives a stronger nonexistence result, suggesting that the Sobolev-embedding
threshold in their projection theorem should not be sharp for \(k>1\). Our Euclidean
argument follows this suggestion: we reformulate the Peres--Schlag proof using
Littlewood--Paley pieces and then apply Bourgain's Kakeya maximal mechanism back
into the nonempty-interior problem.

For any \(f\in L^1_{\mathrm{loc}}(\mathbb R^d)\), let \(f^*_\delta\) be the Kakeya
maximal function with width \(\delta\). Let \(K_d(p,\delta)\) be the sharp constant
for which
\begin{align}\label{KakeyaConj}
    \|f^*_\delta\|_{L^p(G(d,1))}
    \leq
    K_d(p,\delta)\|f\|_{L^p(\mathbb R^d)}
\end{align}
holds, where \(p\geq 1\) and \(0<\delta<1\). For fixed \(d\) and \(p\), let \(\boxdot=\boxdot(d,p)>0\) be any exponent such that
\begin{align}\label{blackboxxx}
     K_d(p,\delta)\lesssim \delta^{-\boxdot},
\end{align}
where the implicit constant depends on $p,d$.
Kakeya maximal conjecture asserts that $\boxdot>0$ can be chosen arbitrarily small when $p=d$.

Bourgain's method relates estimates of the form \eqref{blackboxxx} to
frequency-localized maximal plane transforms. Combining this with the
Littlewood--Paley reformulation described above gives the following Euclidean
projection theorem.

\begin{thm}\label{nonemptyimprove}
Let \(1\leq k\leq n-1\), and assume that \eqref{blackboxxx} holds in dimension \(d=k+1\) with index \(p\)
and exponent \(\boxdot\) in \eqref{blackboxxx}.
Let \(E\subset\mathbb R^n\) be a Borel set with
\[
    \dim_{\mathrm H}E>
    \frac{\boxdot p+n(p-2)+(k+1)}{p-1}.
\]
Then for \(\gamma_{n,k}\)-a.e. \(V\in G(n,k)\), the projection
\(\Pi_V(E)\) has nonempty interior in \(V\).
\end{thm}

Compared with the Peres--Schlag threshold, Theorem~\ref{nonemptyimprove} gives an
improvement only in certain parameter ranges. Even under the full Kakeya maximal
conjecture, it does not improve Theorem~\ref{thm:peres-schlag-euclidean} for all
\((n,k)\). Nevertheless, it gives new Euclidean nonempty-interior results, as illustrated
by Corollary~\ref{improvedcor} below.

We record the current best known Kakeya maximal estimates in several low-dimensional cases \cite{BCCLXZ,MR4521046,MR3868003,MR4201413,MR1363209} and apply them via Theorem \ref{nonemptyimprove}. The sharp bound for $d=2$ was obtained in \cite{MR0447949}, while the case $d\geq7$ was treated in \cite{MR4521046}. We remark that the maximal Kakeya inequalities in \cite{BCCLXZ,MR4521046,MR3868003,MR4201413} are formulated in the dual form of (\ref{KakeyaConj}). The equivalence of (\ref{KakeyaConj}) and its dual formulation is discussed in \cite[Section~22.1]{MR3617376}.

\renewcommand{\arraystretch}{1.5}
 \begin{figure}[H]
\begin{tabular}{ ||c|c|c|| } 
 \hline
  $d=$ & $p\leq$&  \\ 
   \hline 
  3 &  ${(5-3\varepsilon)}/{(2-3\varepsilon)}$ &   Katz--Zahl~\cite{MR3868003,MR4201413}   \\
 4 &    $(159+\sqrt{145})/56$ &     Borges--Chan--Chen--Liu--Xi--Zhan~\cite{BCCLXZ} \\  
 5 &  $18/5$ & Hickman--Rogers-Zhang~\cite{MR4521046}      \\ 
 6 &   $4$ &   Wolff~\cite{MR1363209}  \\ 
 7 &   $34/7$ &   Hickman--Rogers--Zhang~\cite{MR4521046}  \\ 
  \hline
\end{tabular}
    \caption{The state-of-the-art for the Kakeya maximal estimate in the form of (\ref{KakeyaConj}) in low dimensions with $\boxdot=(d-1-d/p')+O(\varepsilon)$ for arbitrary small $\varepsilon>0$, where $1/p+1/p'=1$.}
    \label{maximal}
    \end{figure}
    
A model case is \((n,k)=(4,2)\), where the Peres--Schlag threshold \(\dim_{\mathrm H}E>4\) is vacuous. Using the current best three-dimensional Kakeya maximal estimate, Theorem~\ref{nonemptyimprove}
with \(p=(5-3\epsilon)/(2-3\epsilon)\) and $\boxdot=1/5+O(\epsilon)$, for any arbitrarily small $\epsilon>0$, gives the following result.

\begin{cor}\label{improvedcor}
       Let \(E\subset\mathbb R^4\) be a Borel set with $\dim_{\mathrm H} E>{11}/{3}.$
    Then for \(\gamma_{4,2}\)-a.e.
\(V\in G(4,2)\), the orthogonal projection of $E$ has
nonempty interior in \(V\).
\end{cor}

The finite-field threshold, Theorem \ref{thm:main-projection}, suggests the following Euclidean conjecture.

\begin{conjecture}
\label{conj:l-plus-one-interior-threshold}~
Let \(1< l<n\), and let \(E\subset \mathbb R^n\) be a Borel set.  If $\dim_{\mathrm H}E>l+1$,
then, for almost every \(V\in G(n,l)\),
$\operatorname{int}_{V}\Pi_V(E)\neq\varnothing.$
\end{conjecture}

This conjecture should be regarded as substantially more difficult than its
finite-field analogue.  The proof of Theorem~\ref{thm:main-projection} relies on
a finite-field Kakeya-type theorem, proved by the polynomial method.  In the Euclidean setting, the corresponding Kakeya phenomena remain widely open, and
one should therefore not expect the finite-field argument to transfer directly.

Finally, we show that the exponent \(l+1\) is the smallest possible threshold in the next proposition.

\begin{prop}
\label{thm:l-plus-one-obstruction}~
Let \(1\le l<n\).  There exists a compact set
\(E\subset \mathbb R^n\) with $\dim_{\mathrm H}E=l+1$
such that, for every \(V\in G(n,l)\), $
    \operatorname{int}_{V}\Pi_V(E)=\varnothing $.
\end{prop}

\subsection*{Notation}

Throughout the paper, the integers \(n,m,l,k\) are fixed whenever they appear, and all finite-field
asymptotic notation is taken in the limit \(q\to\infty\).  We write $ A\lesssim_{n,m}B$
to mean that there exists a constant \(C_{n,m}>0\), depending only on \(n\) and
\(m\), such that \(A\le C_{n,m}B\).  Similarly,
\(A\gtrsim_{n,m}B\) means \(B\lesssim_{n,m}A\).  We write $ A\simeq_{n,m}B$
if both \(A\lesssim_{n,m}B\) and \(B\lesssim_{n,m}A\) hold.  When the dependence
of constants is clear from context, we suppress the subscripts and simply write
\(\lesssim\), \(\gtrsim\), \(\simeq\), and \(O(\cdot)\).  We write \(o_q(1)\) for
a quantity tending to \(0\) as \(q\to\infty\).

\subsection*{Organization}

The paper is organized as follows.  In Section~\ref{sec:almost-m-plane-kakeya},
we prove the stability version of the projective vanishing lemma and deduce
Theorem~\ref{thm:alpha-almost-kakeya}.  In
Section~\ref{sec:finite-field-peres-schlag}, we prove
Theorem~\ref{thm:main-projection} and Proposition~\ref{thm:projection-sharpness}.
In Section~\ref{sec:continuous-construction}, we prove
Theorem~\ref{nonemptyimprove}. Finally, in Section~\ref{section:Euclidean_threshold_example}, we prove Proposition~\ref{thm:l-plus-one-obstruction}


\section{Almost $m$-plane Kakeya sets}
\label{sec:almost-m-plane-kakeya}

In this section, we prove
Theorem~\ref{thm:alpha-almost-kakeya}.  This result is the Kakeya-type input
needed later for the sharp finite-field projection theorem. Throughout this section, \(n\) and \(m\) are fixed integers with $1\le m<n,$ and q is a sufficiently large number. We reserve \(m\) for the dimension of the planes and use \(M\) for the multiplicity parameter.

The proof follows the multiplicity-polynomial method of Dvir--Kopparty--Saraf--Sudan \cite{MR3143848} and
Ellenberg--Oberlin--Tao \cite{MR2604979}.  The only point where the argument differs from the
usual \(m\)-plane Kakeya proof is the final projective step.  In the usual
argument, the top-degree homogeneous part of a polynomial vanishes on every
\(m\)-dimensional direction; here it only vanishes on an \(\alpha\)-fraction of
the directions.  We therefore first prove a stability version of the projective
vanishing lemma.

We begin with the projective ingredient.  Given a nonzero homogeneous polynomial
\(Q\), let \(\alpha_s(Q)\) denote the proportion of \(s\)-dimensional subspaces
on which \(Q\) vanishes identically.  The next lemma shows that if
\(\alpha_m(Q)\) is large, then the degree $D$ of \(Q\) must be large.  This is the
precise substitute for the projective vanishing lemma in the almost Kakeya
setting.

\begin{lem}
\label{lem:stability-projective-vanishing}~
Let $Q\in \mathbb F_q[x_1,\dots,x_n]$
be a nonzero homogeneous polynomial of degree \(D\). For \(1\le s\le n\), define
\[
    \alpha_s(Q)
    :=
    \frac{
    \left|
    \left\{
        V\in G(n,s)(\mathbb F_q): Q|_V\equiv 0
    \right\}
    \right|
    }{
    |G(n,s)(\mathbb F_q)|
    },
\]
in which \(Q|_V\equiv 0\) means that the restriction of \(Q\) to \(V\) is the zero
polynomial on \(V\), not merely the zero function on \(V(\mathbb F_q)\).

Then, for every \(1\le m<n\),
\begin{equation}
    \alpha_m(Q)
    \le
    D\sum_{j=m+1}^{n}\frac1{N_j},\quad\text{where}\,\,N_j=\frac{q^j-1}{q-1}.
    \label{eq:alpha-m-bound}
\end{equation}
Consequently, if $\alpha_m(Q)\ge \alpha$,
then
\begin{equation}
    D
    \ge
    \alpha
    \left(
        \sum_{j=m+1}^{n}\frac1{N_j}
    \right)^{-1}.
    \label{eq:stability-degree-exact}
\end{equation}
In particular, for fixed \(m,n\), if \(\alpha=1-o_q(1)\), then
\begin{equation}
    D
    \ge
    (1-o_q(1))\frac{q^{m+1}-1}{q-1}.
    \label{eq:stability-degree-asymptotic}
\end{equation}
\end{lem}

\begin{proof} [Proof of Lemma~\ref{lem:stability-projective-vanishing}]
We will prove the recursive estimate
\begin{equation}
    \alpha_s(Q)
    \le
    \alpha_{s+1}(Q)+\frac{D}{N_{s+1}}
    \label{eq:alpha-recursion}
\end{equation}
for every \(1\le s<n\). Iterating \eqref{eq:alpha-recursion} from
\(s=m\) to \(s=n-1\) will then give \eqref{eq:alpha-m-bound}.

Fix \(1\le s<n\). Choose a random pair $V\subset W\subset \mathbb F_q^n$
as follows: first choose \(W\in G(n,s+1)(\mathbb F_q)\) uniformly, and then
choose \(V\in G(W,s)(\mathbb F_q)\) uniformly among the hyperplanes of \(W\).
By symmetry, the resulting \(V\) is uniformly distributed in
\(G(n,s)(\mathbb F_q)\), and hence,
\begin{equation}
    \alpha_s(Q)=\mathbb P(Q|_V\equiv 0).
    \label{eq:alpha-as-probability}
\end{equation}
We now condition on \(W\). By the property of the conditional probability and (\ref{eq:alpha-as-probability}), $\alpha_s(Q)$ is equal to
\begin{align*}
    &\frac{1}{|G(n,s+1)(\mathbb F_q)|}\sum_{W\in G(n,s+1)(\mathbb F_q)}\mathbb P(Q|_V\equiv 0\mid W)\notag\\
    &=\frac{1}{|G(n,s+1)(\mathbb F_q)|}\sum_{W\in G(n,s+1)(\mathbb F_q)}\frac{\left|\left\{V\in G(n,s)(\mathbb F_q):V\subset W,\,Q|_V\equiv 0\right\}\right|}{N_{s+1}}\notag\\
     &=\frac{1}{|G(n,s+1)(\mathbb F_q)|}\sum_{W\in G(n,s+1)(\mathbb F_q)}\frac{\left|\left\{V\in G(n,s)(\mathbb F_q):V\subset W,\,Q|_V\equiv 0 \land Q|_W\equiv 0\right\}\right|}{N_{s+1}}\notag\\
      &+\frac{1}{|G(n,s+1)(\mathbb F_q)|}\sum_{W\in G(n,s+1)(\mathbb F_q)}\frac{\left|\left\{V\in G(n,s)(\mathbb F_q):V\subset W,\,Q|_V\equiv 0  \land Q|_W\not\equiv 0\right\}\right|}{N_{s+1}},
\end{align*}
where the second equality follows from the fact that the cardinality of the hyperplanes contained in a given $s+1$-dimensional subspace $W\in G(n,s+1)$ is $N_{s+1}$.

From the definition of $\alpha_{s+1}(Q)$, we have $\alpha_{s+1}(Q)$ is 
\begin{align*}
    &\frac{1}{|G(n,s+1)(\mathbb F_q)|}\sum_{W\in G(n,s+1)(\mathbb F_q)}\frac{\left|\left\{V\in G(n,s)(\mathbb F_q):V\subset W,\,Q|_V\equiv 0 \land Q|_W\equiv 0\right\}\right|}{N_{s+1}}\\
    &=\frac{1}{|G(n,s+1)(\mathbb F_q)|}\sum_{W\in G(n,s+1)(\mathbb F_q)}\frac{\left|\left\{V\in G(n,s)(\mathbb F_q):V\subset W,\, Q|_W\equiv 0\right\}\right|}{N_{s+1}},
\end{align*}
and hence, it remains to estimate 
\begin{align}\label{needed}
    \frac{1}{|G(n,s+1)(\mathbb F_q)|}\sum_{W\in G(n,s+1)(\mathbb F_q)}\frac{\left|\left\{V\in G(n,s)(\mathbb F_q):V\subset W,\,Q|_V\equiv 0  \land Q|_W\not\equiv 0\right\}\right|}{N_{s+1}}.
\end{align}
In this situation, \(Q|_W\) is a nonzero homogeneous
polynomial of degree at most \(D\) on the \((s+1)\)-dimensional vector space
\(W\). We claim that \(Q|_W\) can vanish identically on at most \(D\) hyperplanes of
\(W\). Indeed, let \(V=\ker \ell\) be a hyperplane of \(W\), where
\(\ell\in W^\ast\) is nonzero. The restriction of \(Q|_W\) to \(V\) is zero as a
polynomial precisely when the image of \(Q|_W\) in
\(
    \operatorname{Sym}(W^\ast)/(\ell)
\)
is zero. After choosing coordinates on \(W\) such that
\(
    \ell=x_{s+1}
\)
and
\(
    V=\{x_{s+1}=0\},
\)
we have that the condition \((Q|_W)|_V\equiv 0\) means precisely that all monomials of
\(Q|_W\) with \(x_{s+1}\)-exponent \(0\) have zero coefficient. Hence, every
monomial of \(Q|_W\) is divisible by \(x_{s+1}\). Thus \(x_{s+1}\), equivalently
\(\ell\), divides \(Q|_W\).

Distinct hyperplanes of \(W\) correspond to non-proportional linear forms in
\(W^\ast\). Therefore, if \(Q|_W\) vanished identically on \(t\) distinct
hyperplanes of \(W\), then \(Q|_W\) would be divisible by a product of \(t\)
pairwise non-proportional linear forms. It turns out that
\(
    t\le D.
\)
Therefore, the term (\ref{needed}) is dominated by 
\begin{equation}
    \frac{1}{|G(n,s+1)(\mathbb F_q)|}\sum_{W\in G(n,s+1)(\mathbb F_q)}\frac{D}{N_{s+1}}=\frac{D}{N_{s+1}},
    \label{eq:conditional-bound}
\end{equation}
which leads to, by (\ref{eq:conditional-bound}), we obtain the desired recursive formula \eqref{eq:alpha-recursion} and hence, \eqref{eq:alpha-m-bound}. 
If \(\alpha_m(Q)\ge \alpha\), then \eqref{eq:alpha-m-bound} implies
\[
    D
    \ge
    \alpha
    \left(
        \sum_{j=m+1}^{n}\frac1{N_j}
    \right)^{-1},
\]
which proves \eqref{eq:stability-degree-exact}. Finally, suppose that \(m,n\) are fixed and \(q\to\infty\). Since
\[
    \frac1{N_j}
    =
    q^{-(j-1)}\bigl(1+O_j(q^{-1})\bigr).
\]
Therefore, we obtain
\begin{equation}
    \sum_{j=m+1}^{n}\frac1{N_j}
    =
    \frac1{N_{m+1}}
    \left(1+O_{m,n}(q^{-1})\right);
    \label{eq:sum-Nj-asymptotic}
\end{equation}
thus, if \(\alpha=1-o_q(1)\), then \eqref{eq:stability-degree-exact} and
\eqref{eq:sum-Nj-asymptotic} give
\[
    D
    \ge
    (1-o_q(1))N_{m+1}
    =
    (1-o_q(1))\frac{q^{m+1}-1}{q-1},
\]
which is \eqref{eq:stability-degree-asymptotic} and the proof is complete.
\end{proof}

We now use Lemma~\ref{lem:stability-projective-vanishing} to prove the lower bound for \(\alpha\)-almost
\(m\)-plane Kakeya sets.

\begin{proof}[Proof of Theorem~\ref{thm:alpha-almost-kakeya}]
We first dispose of the trivial range of \(\alpha\). If $
    \alpha\le C_{0,n,m}q^{1-m}$,
for a sufficiently large constant \(C_{0,n,m}\), then, after increasing
\(C_{n,m}\) if necessary, $C_{n,m}\alpha^{-1}q^{n-m+1}\ge q^n.$
Thus, the desired estimate is trivial. Henceforth, we may, without loss of generality, assume that
\begin{equation}
    \alpha>C_{0,n,m}q^{1-m}.
    \label{eq:nontrivial-alpha-range}
\end{equation}
We next choose the multiplicity parameter.  The quantity \(A_{n,m}\) below is
the reciprocal of the sum appearing in
Lemma~\ref{lem:stability-projective-vanishing}; it is the degree scale forced
by the stability lemma.
For $1\leq m<n$, define
\begin{equation}
    A_{n,m}
    :=
    \left(
        \sum_{j=m+1}^{n}\frac1{N_j}
    \right)^{-1}.
    \label{eq:Anm-definition}
\end{equation}
By \eqref{eq:sum-Nj-asymptotic}, we have $A_{n,m}=q^m\left(1+O_{n,m}(q^{-1})\right).$
Choose a constant \(c_{n,m}>0\) sufficiently small so that, for all sufficiently
large \(q\), \begin{align}\label{eq:choice-cnm}
c_{n,m}q^m<\frac12 A_{n,m}.\end{align}
Set the constant $M:=\left\lfloor c_{n,m}\alpha q^{m-1}\right\rfloor.$ 
After choosing \(C_{0,n,m}\) sufficiently large, the assumption
\eqref{eq:nontrivial-alpha-range} ensures that
$M\ge 1.$

Now, we prove the lower bound for \(|K|\) by contradiction. Suppose that
\begin{equation}
    |K|
    <
   \frac{\binom{Mq+n-1}{n}}{\binom{M+n-1}{n}}.
    \label{eq:contradict-small-K}
\end{equation}
Let $\mathcal P_{Mq-1}$ be the set of all polynomials $Q$ in $\mathbb F_q[x_1,\dots,x_n]$ in which the degree of $Q$ is at most $Mq-1$, then the dimension of $\mathcal P_{Mq-1}$ is 
\(
    \binom{Mq+n-1}{n}.
\)
Moreover, remark that vanishing to order at least \(M\) at one point imposes at most
\(
    \binom{M+n-1}{n}
\)
linear conditions. Hence, from \cite[Proposition 10]{MR3143848},  \eqref{eq:contradict-small-K} implies that there exists
a nonzero polynomial
$Q\in \mathcal P_{Mq-1}$
which vanishes to order at least \(M\) at every point of \(K\). We write $Q=Q_0+\cdots+Q_d$, where $Q_i$ is the i-th homogeneous component of $Q$ and $d\geq1$ is the degree of the polynomial.

Next, fix \(V\in S_K\). By definition of \(S_K\), there exists
\(a_V\in \mathbb F_q^n\) such that $a_V+V\subseteq K$ and hence, the 
restriction of \(Q\) to the affine \(m\)-plane \(a_V+V\) is a zero polynomial. It turns out that $Q_d|_V\equiv0$. Indeed, for \(v\in V\), the polynomial in one variable
\(
    t\mapsto Q(a_V+tv)
\)
is identically zero, and its coefficient of
\(t^d\) is \(Q_d(v)\). Therefore, \(Q_d(v)=0\) for every
\(v\in V\), which implies that
\(
 S_K\subset
    \left\{
        V\in G(n,m)(\mathbb F_q): Q_d|_V\equiv 0
    \right\}
    .
\)
As a consequence, from the definition of $\alpha$-almost $m$-plane set, we get
\begin{align*}
    \alpha_m(Q_d)\geq \frac{|S_K|}{|G(n,m)|}\geq\alpha,
\end{align*}
and hence, Lemma \ref{lem:stability-projective-vanishing} reveals that degree of $Q_d$ is at least $\alpha A_{n,m}$. Furthermore, by the choice of $M$, we have
\[
\alpha A_{n,m}\leq\deg Q_d\leq\deg Q\leq
    Mq-1<c_{n,m}\alpha q^m,
\]
which is a contradiction, due to the  choice of \(c_{n,m}\) in (\ref{eq:choice-cnm}). 
Therefore,
\[
|K|
    \geq
    \frac{\binom{Mq+n-1}{n}}{\binom{M+n-1}{n}}.
\]
Since \(M\sim_{n,m}\alpha q^{m-1}\), this implies that
$|K|\geq q^n-C_{n,m}\alpha^{-1}q^{n-m+1}$, as desired.
\end{proof}


\section{A sharp finite-field Peres--Schlag orthogonal projection theorem}
\label{sec:finite-field-peres-schlag}

In this section, we prove Theorem~\ref{thm:main-projection} and Proposition~\ref{thm:projection-sharpness}. The main point is that the failure of fullness for an
\(l\)-dimensional projection can be rephrased as a Kakeya-type condition for the
complement.  As mentioned in the Introduction, if \(\Pi_V(E)\) is not full, then some coset of \(V^\perp\)
is disjoint from \(E\), and hence, lies entirely inside
\(\mathbb F_q^n\setminus E\).  Thus, the set of bad projection directions gives a
family of affine \((n-l)\)-planes contained in the complement of \(E\).  The
$\alpha$-almost \(m\)-plane Kakeya set theorem from Section~\ref{sec:almost-m-plane-kakeya}
then forces the complement to be large, or equivalently forces \(E\) to be small.

\begin{proof}[Proof of Theorem~\ref{thm:main-projection}]
Let
\[
    \mathcal{B}:=\left\{
        V\in G(n,l)(\mathbb F_q):
        \Pi_V(E)\neq \mathbb F_q^n/V^\perp
    \right\}
\]
be the exceptional set of bad \(l\)-dimensional directions. Write
$|\mathcal B|=\alpha |G(n,l)(\mathbb F_q)|$ and we assume that, without loss of generality, $\alpha>0$. We will show that \(\alpha=o_q(1)\).

Set $K:=\mathbb F_q^n\setminus E$ and remark that if \(V\in\mathcal B\), then  there exists at least one coset of
\(V^\perp\) which does not meet \(E\), which implies that there exists $a_V\in\mathbb F^n_q$ such that $(a_V+V^\perp)\cap E=\varnothing$.
Hence,
\begin{equation}
    a_V+V^\perp\subseteq K.
    \label{eq:bad-projection-coset-in-complement}
\end{equation}
The map $V\mapsto V^\perp$
is a bijection from \(G(n,l)(\mathbb F_q)\) to
\(G(n,n-l)(\mathbb F_q)\). Therefore, \eqref{eq:bad-projection-coset-in-complement}
shows that \(K\) is
an \(\alpha\)-almost \((n,n-l)\)-Kakeya set.

We may now apply the almost Kakeya lower bound proved in the previous section.
Applying Theorem~\ref{thm:alpha-almost-kakeya} with \(m=n-l\), we obtain
\begin{equation}
    q^n-|K|
    \le
    C_{n,l}\alpha^{-1}q^{n-(n-l)+1}
    =
    C_{n,l}\alpha^{-1}q^{l+1}.
    \label{eq:projection-complement-bound}
\end{equation}
Since \(K=\mathbb F_q^n\setminus E\), then $q^n-|K|=|E|$ and hence, from the estimate (\ref{eq:projection-complement-bound}),
\begin{equation}
    \alpha
    \le
    C_{n,l}\frac{q^{l+1}}{|E|}.
    \label{eq:projection-alpha-bound}
\end{equation}
Finally, the assumed lower bound on \(|E|\) forces the exceptional proportion to
go to zero.  Combining (\ref{eq:projection-alpha-bound}) with the assumption that \(|E|\gtrsim q^s\) with \(s>l+1\), we get
$\alpha\lesssim_{n,l}q^{l+1-s}=o_q(1)$.
As a result, we have
\[
    |\mathcal B|
    =
    o_q(1)|G(n,l)(\mathbb F_q)|.
\]
This proves Theorem~\ref{thm:main-projection}.
\end{proof}

We now turn to the sharpness construction. The same correspondence between
projection failure and affine planes in the complement is now used in the
opposite direction.

\begin{proof}[Proof of Proposition~\ref{thm:projection-sharpness}]
Set
\(
    k:=n-l.
\)
Since \(2\le l<n-1\), we have
\(
    1<k<n.
\)
By the construction described in Remark~4.19 of Ellenberg--Oberlin--Tao \cite{MR2604979}, there
exists a \(k\)-plane Kakeya set
\(
    K\subseteq \mathbb F_q^n
\)
such that \(K\) contains a translate of every \(k\)-dimensional linear subspace
of \(\mathbb F_q^n\), and
\begin{equation}
    |K|
    =
    q^n\left(
        1-(1-2^{-n+k})q^{1-k}+O(q^{-k})
    \right),
    \label{eq:EOT-k-plane-Kakeya-size}
\end{equation}
where the implicit constant depends only on \(n\) and \(k\).
Define
\(
    E:=\mathbb F_q^n\setminus K.
\)
Then by \eqref{eq:EOT-k-plane-Kakeya-size},
\begin{equation}
    |E|
    =
    q^n-|K|
    =
    (1-2^{-n+k})q^{n+1-k}+O(q^{n-k}).
    \label{eq:E-complement-size-k}
\end{equation}
Since \(k=n-l\), \eqref{eq:E-complement-size-k} becomes
\[
|E|
    =
    (1-2^{-l})q^{l+1}+O(q^l).
    \label{eq:E-complement-size-l}
\]
It remains to show that no \(l\)-dimensional projection of \(E\) is full. Given any
\(
    V\in G(n,l)(\mathbb F_q),
\)
there exists
\(a_V\in \mathbb F_q^n\) such that
$a_V+V^\perp\subseteq K.$
Therefore,
\(
    (a_V+V^\perp)\cap E=\varnothing.
\)
Equivalently, the coset
\(
    a_V+V^\perp\in \mathbb F_q^n/V^\perp
\)
does not belong to \(\Pi_V(E)\). Hence,
\(
    \Pi_V(E)\neq \mathbb F_q^n/V^\perp.
\)
This proves Proposition~\ref{thm:projection-sharpness}.
\end{proof}


\section{An improvement of the Peres--Schlag orthogonal projection theorem in Euclidean space}
\label{sec:continuous-construction}

In this section, we prove Theorem~\ref{nonemptyimprove}. The argument is a Littlewood--Paley reformulation of the Peres--Schlag proof, with Bourgain's Kakeya maximal mechanism inserted at the frequency-localized level. More precisely, each projected frequency piece is controlled by a signed maximal \((n-k)\)-plane transform, which Bourgain's reduction bounds by a Kakeya maximal operator in dimension \(k+1\). The resulting estimates are summed using the Riesz energy of a Frostman measure, which shows that for almost every direction, the projected measure has a continuous density.

We first establish the required maximal-transform estimate.
We use a signed frequency-localized version of the maximal plane transform. For
\(1\le m<n\), \(g\in \mathcal S(\mathbb R^n)\), and \(V\in G(n,m)\), set
\[
    g_V(x):=\int_{x+V}g(y)\,d\mathcal H^m(y),
    \qquad x\in V^\perp.
\]
If \(g\) is frequency supported on \(|\xi|\simeq 2^j\), define
\[
    M_j^m g(V)
    :=
    \sup_{x\in V^\perp}
    \left|
        \int_{x+V}g(y)\,d\mathcal H^m(y)
    \right|
    =
    \|g_V\|_{L^\infty(V^\perp)}.
\]
When \(g=\mu_j\) is a Littlewood--Paley piece at frequency \(2^j\), we write
\(M^m(\mu_j)\) for \(M_j^m(\mu_j)\).

We shall use the following annular \(L^2\)-estimate for signed plane transforms,
which is \cite[(24.5)]{MR3617376}.

\begin{prop}
Let \(R>0\), and let \(g\in \mathcal S(\mathbb R^n)\) have frequency support on
\(|\xi|\simeq R\). Then, for every \(1\le m<n\),
\begin{align}\label{integralformula1}
    \iint_{G(n,m)\times V^\perp}
    |g_V(v)|^2\,d\mathcal H^{n-m}(v)\,d\gamma_{n,m}(V)
    \simeq
    R^{-m}\|g\|_{L^2(\mathbb R^n)}^2.
\end{align}
\end{prop}

We next recall the form in which Bourgain's comparison is used. Let
\(V\in G(n,n-k-1)\), put \(S_V:=V^\perp\cap S^{n-1}\), and for \(e\in S_V\) set
\[
    V_e:=V+\mathbb R e\in G(n,n-k).
\]
If \(g\) is frequency localized at scale \(2^j\), then
\[
    \int_{a+V_e}g\,d\mathcal H^{n-k}
    =
    \int_{a+\mathbb R e}g_V(z)\,d\mathcal H^1(z),
    \qquad a\in V_e^\perp.
\]
Here \(g_V\) is a function on \(V^\perp\simeq\mathbb R^{k+1}\), and it is frequency
localized at scale \(\lesssim 2^j\). By \cite[Lemma~24.5 and the proof of
(24.7)]{MR3617376},
\[
    M_j^{n-k}g(V_e)
    \lesssim
    (g_V)^*_{2^{-j}}(e),
\]
where the right-hand side is the usual line Kakeya maximal function of width
\(2^{-j}\) on the \((k+1)\)-dimensional space \(V^\perp\).

\begin{lem}\label{dictionary}
Let \(1\leq k\leq n-1\), and let \(\mu\) be a compactly supported Borel probability
measure on \(\mathbb R^n\). Assume that \(\mu\) is \(\alpha\)-Frostman. Let
\(\mu_j\) denote a Littlewood--Paley piece of \(\mu\) at frequency \(|\xi|\simeq 2^j\).
Assume the \((k+1)\)-dimensional Kakeya maximal inequality holds for
\eqref{blackboxxx} with index \(p\) and some \(\boxdot>0\). Then, for every \(t\)
satisfying
\begin{align}\label{threshold1}
    t>\boxdot p+(n-\alpha)(p-2)+(k+1),
\end{align}
one has
\begin{align}\label{reformulation}
    \left\|
        \sum_{j\geq0}M^{n-k}(\mu_j)
    \right\|_{L^p(G(n,n-k))}
    \lesssim_{\alpha,n,k}
    I_t(\mu)^{1/p}.
\end{align}
In particular, taking \(t=\alpha\), \eqref{reformulation} is available provided
\begin{align}\label{routine}
    \alpha>
    \frac{\boxdot p+n(p-2)+(k+1)}{p-1}.
\end{align}
\end{lem}

\begin{proof} [Proof of Lemma \ref{dictionary}]
Fix \(j\ge0\), and write \(\delta(j)=2^{-j}\). By the parametrization of
\(G(n,n-k)\),
\[
    \|M^{n-k}(\mu_j)\|_{L^p(G(n,n-k))}^p
    \simeq
    \iint_{G(n,n-k-1)\times S_V}
    |M^{n-k}(\mu_j)(V_e)|^p
    \,d\sigma_V(e)\,d\gamma_{n,n-k-1}(V).
\]
The preceding comparison and the \((k+1)\)-dimensional Kakeya maximal inequality
on \(V^\perp\simeq\mathbb R^{k+1}\) imply, for each \(V\in G(n,n-k-1)\),
\begin{align*}
    \int_{S_V}|M^{n-k}(\mu_j)(V_e)|^p\,d\sigma_V(e)
    &\lesssim
    \int_{S_V}|((\mu_j)_V)^*_{\delta(j)}(e)|^p\,d\sigma_V(e)\\
    &\lesssim
    \delta(j)^{-\boxdot p}
    \int_{V^\perp}|(\mu_j)_V(v)|^p\,d\mathcal H^{k+1}(v).
\end{align*}
Consequently, the $L^p$-norm of $M^{n-k}(\mu_j)$ can be dominated by
\begin{align}\label{Lpformaximal}
    \delta(j)^{-\boxdot p}
    \iint_{G(n,n-k-1)\times V^\perp}
    |(\mu_j)_V(v)|^p\,
    d\mathcal H^{k+1}(v)\,d\gamma_{n,n-k-1}(V).
\end{align}

Since \(\mu\) is \(\alpha\)-Frostman, we have
\(
    \|\mu_j\|_\infty\lesssim C_\alpha 2^{j(n-\alpha)}.
\)
Moreover, by the rapid decay of the Littlewood--Paley kernel and a standard
tube-covering estimate for Frostman measures, uniformly in \(V\in G(n,n-k-1)\),
\[
    \|(\mu_j)_V\|_{L^\infty(V^\perp)}
    \lesssim C_\alpha 2^{j(n-\alpha)}.
\]
Hence, \eqref{Lpformaximal} gives
\[
\begin{aligned}
    \|M^{n-k}(\mu_j)\|_{L^p(G(n,n-k))}^p
    &\lesssim
    \delta(j)^{-\boxdot p}
    2^{j(n-\alpha)(p-2)}
    \iint_{G(n,n-k-1)\times V^\perp}
    |(\mu_j)_V(v)|^2\,
    d\mathcal H^{k+1}(v)\,d\gamma_{n,n-k-1}(V).
\end{aligned}
\]
Since \(\mu_j\) is frequency supported on \(|\xi|\simeq2^j\), \eqref{integralformula1}
with \(m=n-k-1\) yields
\begin{align}\label{wrapup}
    \|M^{n-k}(\mu_j)\|_{L^p(G(n,n-k))}^p
    \lesssim
    \delta(j)^{-\boxdot p}
    2^{j(n-\alpha)(p-2)}
    2^{-j(n-k-1)}
    \|\mu_j\|_2^2.
\end{align}
Define
\(
    I_{t,j}(\mu):=2^{j(t-n)}\|\mu_j\|_2^2.
\)
Using \(\delta(j)=2^{-j}\), \eqref{wrapup} becomes
\[
    \|M^{n-k}(\mu_j)\|_{L^p}
    \lesssim
    2^{\frac jp\{\boxdot p+(n-\alpha)(p-2)+(k+1)-t\}}
    I_{t,j}(\mu)^{1/p}.
\]
If \eqref{threshold1} holds, choose \(\eta>0\) such that
\[
    t=\boxdot p+(n-\alpha)(p-2)+(k+1)+\eta.
\]
Then,
\[
    \|M^{n-k}(\mu_j)\|_{L^p}
    \lesssim
    2^{-\eta j/p}I_{t,j}(\mu)^{1/p}.
\]
By H\"older's inequality in \(j\),
\[
    \sum_{j\ge0}2^{-\eta j/p}I_{t,j}(\mu)^{1/p}
    \lesssim_\eta
    \left(\sum_{j\ge0}I_{t,j}(\mu)\right)^{1/p}
    \lesssim I_t(\mu)^{1/p}.
\]
The triangle inequality gives \eqref{reformulation}. Finally, setting \(t=\alpha\)
in \eqref{threshold1} and rearranging gives \eqref{routine}.
\end{proof}

Using this lemma and Frostman's lemma, we obtain the improved Euclidean projection theorem for the nonempty interior problem.

\begin{proof}[Proof of Theorem \ref{nonemptyimprove}]
Set
\[
    \Theta:=\frac{\boxdot p+n(p-2)+(k+1)}{p-1}.
\]
Assume \(\dim_{\mathrm H}E>\Theta\). Choose \(s\) with
\(
    \Theta<s<\dim_{\mathrm H}E.
\)
Then we may choose \(t\) satisfying
\[
    \boxdot p+(n-s)(p-2)+(k+1)<t<s.
\]
By Frostman's lemma, there exists a compactly supported probability measure \(\mu\)
on \(E\) such that \(\mu\) is \(s\)-Frostman. Since \(t<s\), we have
\(
    I_t(\mu)<\infty.
\)
Let \(\{P_j\}_{j\ge1}\), together with a low-frequency projection \(P_{\le0}\), be a
smooth Littlewood--Paley decomposition with radial Fourier multipliers, and set
\(
    \mu_j:=P_j\mu,
\)  and \(
    \mu_{\le0}:=P_{\le0}\mu.
\)
For \(V\in G(n,k)\), define
\[
    F_{V,j}(x):=\int_{x+V^\perp}\mu_j(y)\,d\mathcal H^{n-k}(y),
    \qquad x\in V,
\]
and define \(F_{V,\le0}\) similarly. Then each \(F_{V,j}\) is smooth on \(V\), and
\[
    \|F_{V,j}\|_{L^\infty(V)}
    \le
    M^{n-k}(\mu_j)(V^\perp).
\]
Applying Lemma~\ref{dictionary} with \(\alpha=s\), and using the measure-preserving
bijection \(V\mapsto V^\perp\) from \(G(n,k)\) to \(G(n,n-k)\), we obtain for
\(\gamma_{n,k}\)-a.e. \(V\)
\[
    \sum_{j\ge1}\|F_{V,j}\|_{L^\infty(V)}
    \le
    \sum_{j\ge1}M^{n-k}(\mu_j)(V^\perp)
    <\infty.
\]
Hence the series
\(
    F_{V,\leq 0}+\sum_{j\geq 1}F_{V,j}
\)
converges uniformly on \(V\) to a continuous function, which we denote by \(F_V\).
Since
\[
    P_{\leq 0}\mu+\sum_{j\geq 1}P_j\mu=\mu
\]
in the sense of distributions on \(\mathbb R^n\), testing against functions of the form
\(\varphi\circ \Pi_V\), with \(\varphi\in C_c^\infty(V)\), gives
\(
    (\Pi_V)_\#\mu=F_V\,d\mathcal H^k
\)
as measures on \(V\). Since \((\Pi_V)_\#\mu\) is a nonzero positive measure, the
continuous density \(F_V\) is nonnegative and not identically zero. Hence \(F_V>0\)
on some nonempty open ball in \(V\). As \(\operatorname{spt}\mu\subset E\), we have
\(
    \operatorname{spt}((\Pi_V)_\#\mu)\subset \Pi_V(E),
\)
and therefore
\(
    \operatorname{int}_V\Pi_V(E)\neq\varnothing
\)
for \(\gamma_{n,k}\)-a.e. \(V\in G(n,k)\).
\end{proof}

\section{A Euclidean threshold example} \label{section:Euclidean_threshold_example}
In this section, we prove Proposition~\ref{thm:l-plus-one-obstruction}. This construction is an extension of Mattila's example, as described in Section~6 of Peres--Schlag \cite{MR1749437}.

\begin{proof}[Proof of Proposition \ref{thm:l-plus-one-obstruction}]
Fix $H\in G(n,l+1)$.
We shall construct \(E\) inside \(H\).
By \cite[Theorem 11.1]{MR3617376}, there
exists a Borel set $A\subset H$
of \((l+1)\)-dimensional Lebesgue measure zero such that, for every
one-dimensional linear subspace \(L\subset H\), there exists \(a_L\in H\)
with $a_L+L\subset A.$ In other words, \(A\) contains a whole affine line in every direction in \(H\).

Choose a coordinate system on \(H\cong\mathbb R^{l+1}\), and let $\mathbb Q_H^{l+1}\subset H$
denote the corresponding set of rational points.  Define
\[
    D
    :=
    \bigcup_{q\in \mathbb Q_H^{l+1}}(q+A),
\]
which has
\((l+1)\)-dimensional Lebesgue measure zero in \(H\). Let $E\subset H\setminus D$ be a compact set
with positive \((l+1)\)-dimensional Lebesgue measure, which is available due to the inner compact regular property of the Lebesgue measure. In particular,
\[
    \dim_{\mathrm H}E=l+1.
\]
We now prove that every \(l\)-dimensional orthogonal projection of \(E\)
has an empty interior.  Fix \(V\in G(n,l)\), and consider the restricted
projection
\[
    L_V:=\Pi_V|_H:H\to V.
\]
There are two cases, according to the surjectivity of \(L_V\).  First, suppose that $L_V$ is not surjective.
Then $\Pi_V(E)=L_V(E)\subset L_V(H)$ and \(L_V(H)\) is a proper linear subspace of \(V\).  Therefore, in such a case,
$\operatorname{int}_{V}\Pi_V(E)=\varnothing$.

It remains to consider the case where \(L_V\) is onto. Then \(\ker L_V\) is a
one-dimensional subspace of \(H\). By the defining property of \(A\), there exists
\(a_V\in H\) such that
\(
    a_V+\ker L_V\subset A.
\)
Hence, for every \(q\in Q_H^{l+1}\),
\[
    q+a_V+\ker L_V\subset q+A\subset D.
\]
Since \(E\subset H\setminus D\), the fiber over \(L_V(q+a_V)\) is disjoint from \(E\).
Therefore,
\(
    L_V(q+a_V)\notin L_V(E).
\)
As \(Q_H^{l+1}\) is dense in \(H\) and \(L_V\) is onto, the set
\(
    \{L_V(q+a_V):q\in Q_H^{l+1}\}
\)
is dense in \(V\). Thus \(V\setminus L_V(E)\) is dense in \(V\), and consequently
\(
    \operatorname{int}_V L_V(E)=\operatorname{int}_V \Pi_V(E)=\varnothing.
\)
\end{proof}
\bigskip

\noindent{\bf AI Disclosure:}
Generative AI tools were used solely for language editing and proofreading. All mathematical results, arguments, proofs, and conclusions were developed, verified, and approved by the authors.
\medskip

\noindent {\bf Acknowledgment:}
The first author is supported by the MOE Taiwan-Caltech Fellowship during the conduct of this
research.  The second and third authors are supported by the National Science and Technology Council (NSTC) under Grant No.~111-2115-M-002-010-MY5.


\bigskip

\begin{thebibliography}{10}
\newcommand{\enquote}[1]{`#1'}
\bibitem{MR1097257}
J. Bourgain, {\it Besicovitch type maximal operators and applications to Fourier analysis}, Geom. Funct. Anal. {\bf 1} (1991), no.~2, 147--187.


\bibitem{BCCLXZ}
   T. Borges, T. Chan, M. Chen, D. Liu, Y. Xi and Y. Zhan, {\it Restriction and Kakeya maximal estimates in $\mathbb{R}^4$}, arXiv: 2511.22824 (2025). 
    
\bibitem{MR3753167}
C. Chen, {\it Projections in vector spaces over finite fields},  Ann. Acad. Sci. Fenn. Math. {\bf 43} (2018), no.~1, 171--185.

\bibitem{MR0447949}
A.~J. C\'ordoba, {\it The Kakeya maximal function and the spherical summation multipliers}, Amer. J. Math. {\bf 99} (1977), no.~1, 1--22.

\bibitem{MR2525780}
Z. Dvir, {\it On the size of Kakeya sets in finite fields}, J. Amer. Math. Soc. {\bf 22} (2009), no.~4, 1093--1097.

\bibitem{MR3004132}
Z. Dvir, {\it Incidence theorems and their applications}, Found. Trends Theor. Comput. Sci. {\bf 6} (2010), no.~4, 257--393 (2012).

\bibitem{MR3143848}
Z. Dvir, S. Kopparty, S. Saraf, and M. Sudan, {\it Extensions to the method of multiplicities, with applications to Kakeya sets and mergers}, SIAM J. Comput. {\bf 42} (2013), no.~6, 2305--2328.

\bibitem{MR2604979}
J.~S. Ellenberg, R. Oberlin and T.~C. Tao, {\it The Kakeya set and maximal conjectures for algebraic varieties over finite fields}, Mathematika. {\bf 56} (2010), no.~1, 1--25.

\bibitem{MR0553579}
K.~J. Falconer, {\it Continuity properties of $k$-plane integrals and Besicovitch sets}, Math. Proc. Cambridge Philos. Soc. {\bf 87} (1980), no.~2, 221--226.

\bibitem{MR3558147}
K.~J. Falconer, J.~M. Fraser and X. Jin, {\it Sixty years of fractal projections}, in Fractal geometry and stochastics V, 3--25, Progr. Probab., 70, Birkh\"auser/Springer, Cham.


\bibitem{MR3272924}
L. Guth and N.~H. Katz, {\it On the Erd\H os distinct distances problem in the plane}, Ann. of Math. (2) {\bf 181} (2015), no.~1, 155--190.

\bibitem{MR4521046}
J. Hickman, K.~M.~K. Rogers and R. Zhang, {\it Improved bounds for the Kakeya maximal conjecture in higher dimensions}, Amer. J. Math. {\bf 144} (2022), no.~6, 1511--1560.

\bibitem{MR3868003}
N.~H. Katz and J. Zahl, {\it An improved bound on the Hausdorff dimension of Besicovitch sets in $\mathbb{R}^3$}, J. Amer. Math. Soc. {\bf 32} (2019), no.~1, 195--259.

\bibitem{MR4201413}
N.~H. Katz and J. Zahl, {\it A Kakeya maximal function estimate in four dimensions using planebrushes}, Rev. Mat. Iberoam. {\bf 37} (2021), no.~1, 317--359.

\bibitem{MR0248779}
R.~P. Kaufman, {\it On Hausdorff dimension of projections}, Mathematika {\bf 15} (1968), 153--155.

\bibitem{MR0409774}
P. Mattila, {\it Hausdorff dimension, orthogonal projections and intersections with planes}, Ann. Acad. Sci. Fenn. Ser. A I Math. {\bf 1} (1975), no.~2, 227--244.

\bibitem{MR2044636}
P. Mattila, {\it Hausdorff dimension, projections, and the Fourier transform}, Publ. Mat. {\bf 48} (2004), no.~1, 3--4.

\bibitem{MR3617376}
P. Mattila, {\it Fourier analysis and Hausdorff dimension}, Cambridge Studies in Advanced Mathematics, 150, Cambridge Univ. Press, Cambridge, 2015.

\bibitem{MR0063439}
J.~M. Marstrand, {\it Some fundamental geometrical properties of plane sets of fractional dimensions}, Proc. London Math. Soc. (3) {\bf 4} (1954), 257--302.



\bibitem{MR1749437}
Y. Peres and W. Schlag, {\it Smoothness of projections, Bernoulli convolutions, and the dimension of exceptions}, Duke Math. J. {\bf 102} (2000), no.~2, 193--251.

\bibitem{MR1363209}
T.~H. Wolff, {\it An improved bound for Kakeya type maximal functions}, Rev. Mat. Iberoam. {\bf 11} (1995), no.~3, 651--674.


\bibitem{MR0575692}
R. Zippel, {\it Probabilistic algorithms for sparse polynomials}, in Symbolic and algebraic computation (EUROSAM '79, Internat. Sympos., Marseille, 1979), pp. 216--226, Lecture Notes in Comput. Sci., 72, Springer, Berlin-New York.
\end{thebibliography}
\end{document}